\documentclass[a4paper,UKenglish,cleveref, autoref, thm-restate, nolineno]{socg-lipics-v2021}

\pdfoutput=1 
\hideLIPIcs  

\usepackage[percent]{overpic}
\usepackage{graphicx}

\usepackage{url}

\newcommand{\define}[1]{\textup{\textbf{\boldmath{#1}}}}

\renewcommand{\Im}{\mathrm{Im}} 

\newcommand{\bbR}{\mathbb{R}} 

\renewcommand{\to}{\rightarrow} 
\newcommand{\into}{\hookrightarrow} 
\newcommand{\iso}{\cong} 
\renewcommand{\equiv}{\simeq} 
\newcommand{\id}{\mathrm{id}}

\newcommand{\VR}{\mathsf{VR}}
\newcommand{\DR}{\mathsf{DR}}
\newcommand{\dg}{d_{\mathrm{g}}}
\newcommand{\dE}{d_{\mathrm{E}}}
\newcommand{\rel}{\mathrm{rel} \,}
\newcommand{\innerrad}{R}
\newcommand{\outerrad}{Q}
\newcommand{\weight}{w}
\newcommand{\annulus}{A_{\innerrad, \outerrad}}
\newcommand{\weightannulus}[3]{\mathcal{A}(#1, #2, #3)}
\newcommand{\wannulus}{\weightannulus{\innerrad}{\outerrad}{\weight}}
\newcommand{\wannuluszero}{\weightannulus{\innerrad}{\outerrad}{0}}
\newcommand{\wann}{\mathcal{A}}
\newcommand{\sdense}{a}
\newcommand{\ldense}{b}
\newcommand{\Psimap}{H^{\scriptscriptstyle\VR}}

\title{The Degree-Rips Complexes of an Annulus with Outliers}

\author{Alexander Rolle}{Department of Mathematics, 
Technical University of Munich, Germany}{}{}{}

\authorrunning{A. Rolle} 

\Copyright{Alexander Rolle} 



\acknowledgements{I would like to thank Michael Lesnick 
for helpful conversations about robustness of degree-Rips, 
and Luis Scoccola and Fabian Roll 
for various helpful conversations about topics related to this paper. 
I would also like to thank the reviewers for their constructive comments.}

\EventEditors{Xavier Goaoc and Michael Kerber}
\EventNoEds{2}
\EventLongTitle{38th International Symposium on Computational Geometry (SoCG 2022)}
\EventShortTitle{SoCG 2022}
\EventAcronym{SoCG}
\EventYear{2022}
\EventDate{June 7--10, 2022}
\EventLocation{Berlin, Germany}
\EventLogo{socg-logo.pdf}
\SeriesVolume{224}
\ArticleNo{XX} 

\begin{document}

\maketitle

\begin{abstract}
The degree-Rips bifiltration is the most computable of the 
parameter-free, density-sensitive bifiltrations in topological data analysis. 
It is known that this construction is stable to small perturbations of the input data, 
but its robustness to outliers is not well understood. 
In recent work, Blumberg--Lesnick prove a result in this direction using the Prokhorov distance and homotopy interleavings. 
Based on experimental evaluation, they argue that a more refined approach is desirable, 
and suggest the framework of homology inference. 
Motivated by these experiments, we consider a probability measure that is uniform with high density on an annulus, 
and uniform with low density on the disc inside the annulus. 
We compute the degree-Rips complexes of this probability space up to homotopy type, 
using the Adamaszek--Adams computation of the Vietoris--Rips complexes of the circle. 
These degree-Rips complexes are the limit objects for the Blumberg--Lesnick experiments. 
We argue that the homology inference approach has strong explanatory power in this case, 
and suggest studying the limit objects directly as a strategy for further work.
\end{abstract}

\section{Introduction}

\subsection{Background}

The degree-Rips bifiltration \cite{lesnick-wright} 
is a density-sensitive construction based on 
the Vietoris--Rips filtration. 
The sensitivity to density has two consequences: 
degree-Rips can distinguish metric spaces that are close in the 
Gromov--Hausdorff distance 
but have different patterns of density, 
and degree-Rips is more robust to noise and outliers. 
There are other bifiltrations that share these goals, 
but degree-Rips is of particular interest because, 
using available algorithms and software, 
it is the most computable of these bifiltrations 
that requires only a metric on the data as input. 

If $X$ is a finite metric space, 
the degree-Rips complex $\DR(X)$, 
at parameter $(s, k)$, 
is the full subcomplex of the Vietoris--Rips complex 
$\VR(X)(s)$ on those vertices having degree at least $k-1$ in the one-skeleton. 
Equivalently, we take the Vietoris--Rips complex of the subset 
$X_{(s, k)} = \{ x \in X : |B(x, s)| \geq k \}$, 
where $B(x, s)$ is the open ball in $X$ about $x$ of radius $s$.

There has now been work on the stability of degree-Rips by several authors. 
Recent results of Blumberg--Lesnick \cite{blumberg-lesnick} 
are notable in that they allows for true outliers: 
one can add an arbitrary point to a finite metric space, 
and their results guarantee some relationship between the respective degree-Rips bifiltrations. 
The main result of Blumberg--Lesnick for degree-Rips 
says that if the Gromov--Prokhorov distance between 
the uniform probability measures of two finite metric spaces is less than $\delta$, 
then one has a homotopy interleaving between their degree-Rips bifiltrations, 
with additive term $\delta$, 
and with a multiplicative factor in the Rips parameter $s$. 
They show moreover that the multiplicative factor is tight.

This framework for studying the robustness of degree-Rips is very natural, 
but in the same paper, Blumberg--Lesnick observe that the result 
does not fully capture the robustness of degree-Rips observed in practice. 
They report on the following experiment. 
They consider two pointclouds: a uniform sample of 475 points from an annulus, 
and another pointcloud obtained 
by adding 25 points sampled uniformly from the disc inside the annulus. 
Then they use the RIVET software \cite{rivet} to visualize $H_1 \DR$ of both pointclouds. 
Given the output on the sample with no outliers, 
their results guarantee that a certain region of the degree-Rips parameter space 
for the sample with outliers must have non-zero Hilbert function 
(i.e., the degree-Rips complexes in this region must have non-zero $H_1$); 
however this region is small compared to the observed region where 
the Hilbert function is non-zero. 
It appears that there is a trade-off between the generality of this result, 
and the ability to provide explanatory power in concrete cases such as this one.

We make one more remark before explaining the contribution of this paper. 
The degree-Rips bifiltration is closely related to existing methods for clustering. 
Several widely-used algorithms that arose independently of topological data analysis, 
such as the hierarchical clustering algorithm 
robust single-linkage \cite{chaudhuri-dasgupta-10} 
and the clustering algorithms DBSCAN \cite{dbscan} 
and HDBSCAN \cite{campello-moulavi-sander}, 
can be computed directly from degree-Rips by taking the connected components of the 1-skeleton. 
These algorithms are used in part because of their observed robustness to noise and outliers. 
A satisfactory understanding of the robustness of degree-Rips 
would also add to our understanding of the robustness of these algorithms.

\subsection{Homology inference for degree-Rips}

Motivated by their experiments, 
Blumberg--Lesnick suggest the framework of homology inference 
for obtaining more refined results about the robustness of degree-Rips. 
We now explain one approach to homology inference for degree-Rips. 

There is a natural generalization of the degree-Rips complexes to 
metric probability spaces (\cref{metric-probability-space}). 
Given such a space $(X, \mu)$, 
the degree-Rips complex $\DR(X, \mu)$ at parameter $(s, k)$ 
is the Vietoris--Rips complex of the subset 
$X_{(s, k)} = \{ x \in X : \mu(B(x, s)) \geq k \}$ 
(\cref{def-DR}).
If one gives a finite metric space its uniform probability measure, 
then this definition agrees with the previous one up to normalization. 
Furthermore, on compact metric probability spaces, degree-Rips is $2$-Lipschitz, 
comparing the input using the Gromov--Hausdorff--Prokhorov distance, 
and comparing the output using the homotopy-interleaving distance 
\cite[Theorem 6.5.1]{scoccola-thesis}. 
For the sake of this paper, it is not necessary to know the definition of the 
Gromov--Hausdorff--Prokhorov distance, 
but just the following consequence. 
Say that $\mu$ is a compactly-supported probability measure on Euclidean space 
with support $C$, let $X$ be a finite sample from $\mu$, 
let $\mu_X$ be the uniform measure on $X$, 
and let $\bar{\mu}_X$ be the empirical measure on Euclidean space determined by $X$. 
If the Hausdorff distance $d_{\mathrm{H}}(X, C)$ 
and the Prokhorov distance $d_{\mathrm{P}}(\bar{\mu}_X, \mu)$ are less than $\epsilon$, 
then the homotopy interleaving distance between $\DR(X, \mu_X)$ and $\DR(C, \mu)$ 
is less than $2\epsilon$. 
Here, the hypothesis is stronger than in the result of Blumberg--Lesnick, 
because it includes the Hausdorff hypothesis, 
and the conclusion is also stronger, since one obtains additive interleavings. 
So, we know the limit objects for degree-Rips: 
in probability, $\DR(X, \mu_X)$ converges to $\DR(C, \mu)$ 
in the homotopy-interleaving distance 
as the size of $X$ goes to infinity.

What consequence does this have for the robustness of degree-Rips? 
One way to pose the question of the robustness of degree-Rips is the following. 
If we have a finite metric space $X$, 
and $X'$ has been obtained from $X$ by adding a small number of outliers, 
how do we expect that $\DR(X)$ and $\DR(X')$ are related? 
Roughly speaking, this is how Blumberg--Lesnick ask the question. 
On the level of metric probability spaces, there is an analogous question: 
if we have $\mu$ and $C$ as before, 
and $\mu'$ has been obtained from $\mu$ by mixing with the uniform measure 
on some $C'$ with $C \subset C'$, 
how are $\DR(C, \mu)$ and $\DR(C', \mu')$ related? 

Consider the metric probability spaces from which 
the finite input in the Blumberg--Lesnick experiments are sampled. 
Let $\wannulus$ be the metric probability space 
that consists of the union of the annulus 
$\{ p \in \bbR^2 : \innerrad \leq ||p|| \leq \outerrad \}$ 
and the disc 
$\{ p \in \bbR^2 : ||p|| < \innerrad \}$, 
with a uniform measure on each piece, 
such that the measure of the disc is equal to $w$. 
If $w=0$, take the underlying metric space to be just the annulus. 
See \cref{preliminaries} for a detailed definition. 
In this paper, we compute the degree-Rips bifiltrations of $\wannulus$ 
up to homotopy type, 
using the Adamaszek--Adams computation of the Vietoris--Rips complexes 
of the circle~\cite{adamaszek-adams-published}. 
We now state the result in the case $w > 0$. 
See \cref{sphere-regions} for an 
illustration.\footnote{Scripts to reproduce the figures 
are available at \url{https://github.com/alexanderrolle/degreeRips_annulus}}

\begin{theorem} \label{intro-theorem}
Let $\wannulus$ be a weighted annulus with $w > 0$. 
There are continuous maps 
$\varphi_{\ell} \colon (0, \infty) \to [0,1]$ 
for $\ell = 0, 1, 2, \dots, \infty$ such that, 
for any $s > 0$ and any $k \in [0,1]$,
\[
	\DR(\wannulus)(s,k) \equiv 
	\begin{cases}
	\emptyset & \text{ \; if \; } k > \varphi_0(s) \\
	S^{2\ell+1} & \text{ \; if \; } \varphi_{\ell}(s) > k > \varphi_{\ell+1}(s) 
	\text{ \; for } \ell \neq \infty \\
	* & \text{ \; if \; } \varphi_{\infty}(s) > k
	\end{cases}
\]
Moreover, if $0 < \ell < \infty$ and $0 < s \leq s'$ and $0 \leq k' \leq k \leq 1$ 
are such that
\[
	\varphi_{\ell}(s) > k > \varphi_{\ell+1}(s) \text{ \; and \; } 
	\varphi_{\ell}(s') > k' > \varphi_{\ell+1}(s') \, ,
\]
then the inclusion 
$\DR(\wannulus)(s,k) \into \DR(\wannulus)(s',k')$ 
is a homotopy equivalence.
\end{theorem}

The result for the case $w=0$ is similar, 
but the curves that bound the regions are no longer continuous; 
we state the result in this case in \cref{the-annulus-without-outliers}. 
Varying $w$ does not have much effect on $\DR(\wannulus)$ 
while $w$ remains small and non-zero. 
Setting $w=0$ has a large effect, 
as we see in \cref{sphere-regions}, 
because in this case only points on the annulus are allowed to appear as vertices of degree-Rips.

Comparing these calculations with the results obtained from finite samples, 
we see that the homology inference approach indeed provides strong explanatory power. 
See \cref{H1-regions}. 
The region of the parameter space where $\DR(\wannulus)$ has the homotopy type of $S^1$, 
and thus has rank $1$ homology in dimension $1$, 
is similar to the region where the Hilbert function of the degree-Rips complexes of the sample 
is equal to $1$. 
Note that when we set $w > 0$ and allow for outliers, 
the region where we see non-zero Hilbert function in the sample extends 
a little further in the direction of increasing Rips parameter value. 
For larger values of the Rips parameter, 
outliers begin to appear as vertices in degree-Rips, 
and they create connections between dense regions that would not otherwise appear. 
In $\wannulus$, all points in the inner disc are allowed to appear as vertices, 
and so these connections appear as soon as possible.

\begin{figure}[ht!]
	\centering
	\begin{minipage}[b]{0.45\textwidth}
		\includegraphics[width=\textwidth]{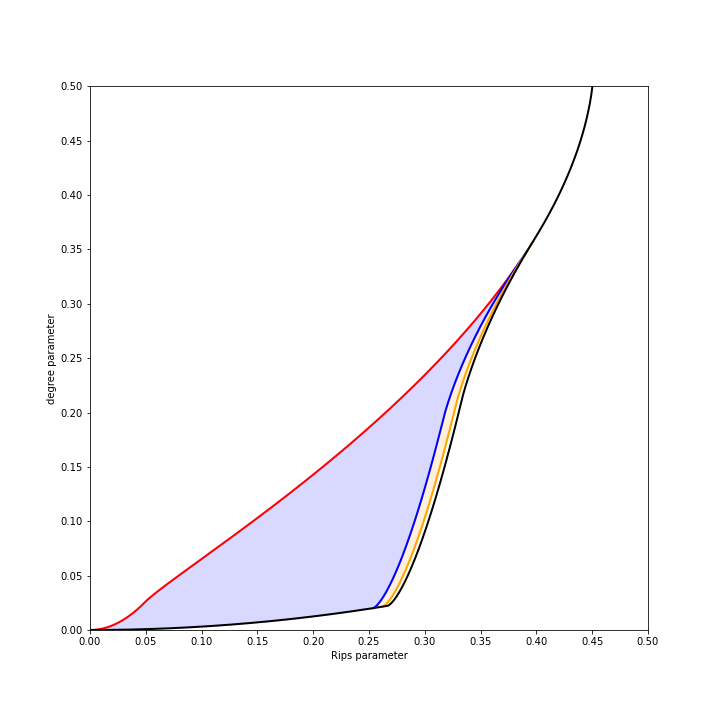}
	\end{minipage}
	\hfill
	\begin{minipage}[b]{0.45\textwidth}
		\includegraphics[width=\textwidth]{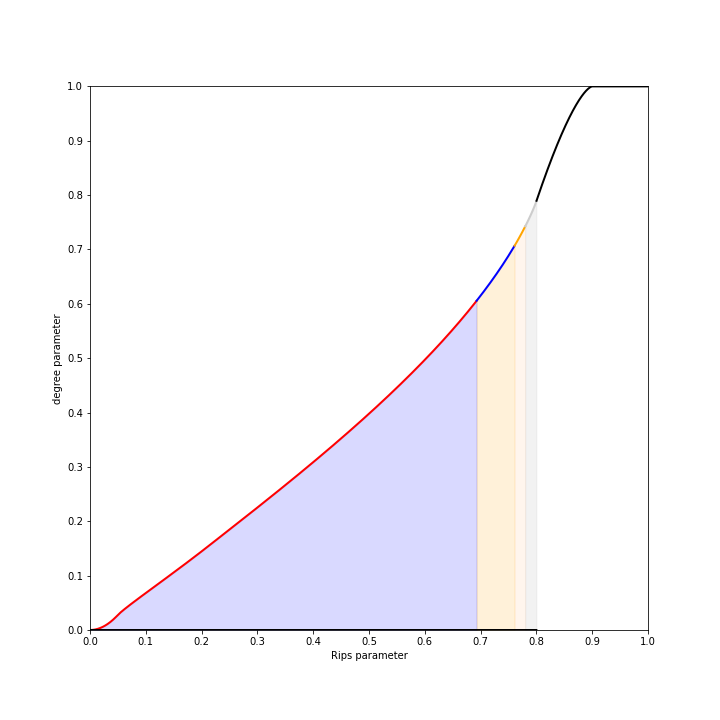}
	\end{minipage}
	\caption{On the left, we consider $\wannulus$ with inner radius 
	$\innerrad=0.4$, outer radius $\outerrad=0.5$, 
	and $w=0.05$, so one expects 25 outliers in a sample of 500, 
	as in the Blumberg--Lesnick experiments. 
	We plot $\varphi_0$ (red), $\varphi_1$ (blue), $\varphi_2$ (yellow), 
	and $\varphi_{\infty}$ (black). 
	The blue region is where the homotopy type of $\DR(\wannulus)$ is $S^1$, 
	and the yellow region is where the homotopy type is $S^3$. 
	On the right, we consider $\wannuluszero$. 
	The meaning of the colors is the same. 
	In this case the boundary curves are not continuous. 
	Note that the two figures are plotted at different scales.}
	\label{sphere-regions}
\end{figure}

\begin{figure}[ht!]
	\centering
	\begin{minipage}[b]{0.45\textwidth}
		\includegraphics[width=\textwidth]{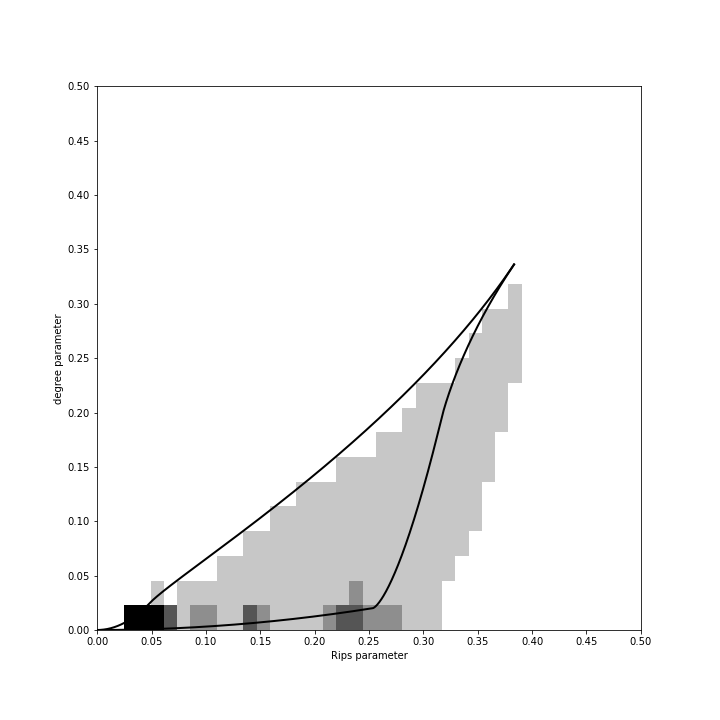}
	\end{minipage}
	\hfill
	\begin{minipage}[b]{0.45\textwidth}
		\includegraphics[width=\textwidth]{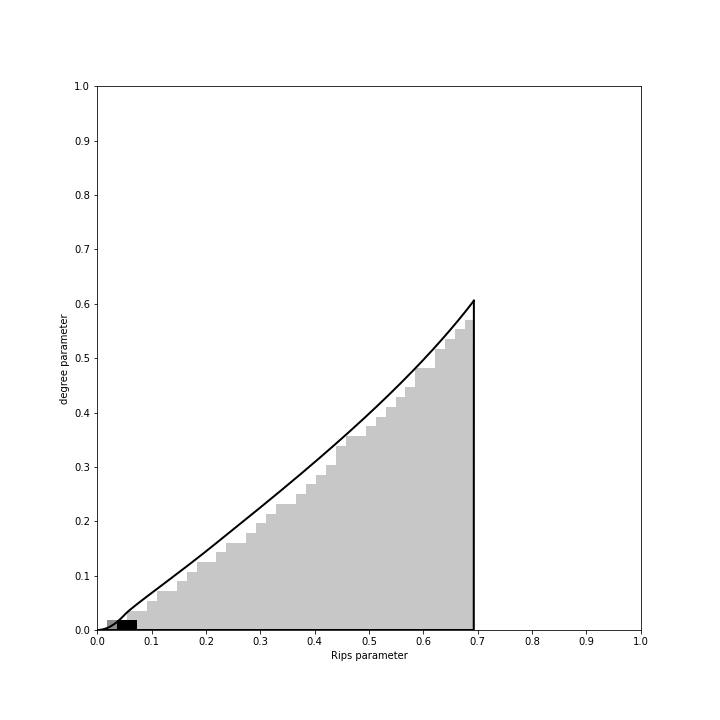}
	\end{minipage}
	\caption{We reproduce the Blumberg--Lesnick experiment. 
	On the left we consider a sample with outliers: 
	we sample 500 points from $\wannulus$, 
	where $\innerrad$, $\outerrad$, and $w$ are as in \cref{sphere-regions}. 
	We use RIVET to compute the Hilbert function of $H_1\DR$ of this sample; 
	the light grey region is where the Hilbert function is equal to $1$, 
	and darker grey corresponds to higher values of the Hilbert function. 
	The border of the $S^1$ region of $\DR(\wannulus)$ is overlayed in black. 
	On the right we consider a sample without outliers: 
	we sample 500 points from $\wannuluszero$. 
	We compare the Hilbert function of  $H_1\DR$ of the sample
	with the $S^1$ region of $\DR(\wannuluszero)$ in the same way.
	Note that the two figures are plotted at different scales.}
	\label{H1-regions}
\end{figure}

\subsection{Related work}

Along with the results mentioned in the introduction, 
Jardine has proved a stability result for 
degree-Rips \cite{jardine-persistent-homotopy}, 
using a hypothesis involving configuration spaces, 
rather than a distance between pointclouds. 
Much work has been done on homology inference, 
using a variety of approaches.  
See for example 
\cite{bcehm-inferring, 
niyogi-smale-weinberger, 
attali-lieutier-salinas, 
zigzag-zoology, 
bobrowski-mukherjee-taylor,
boissonnat_chazal_yvinec}. 
The connection between degree-Rips and existing clustering methods 
was observed by McInnes--Healy \cite{mcinnes-healy}, 
and studied further in 
\cite{jardine-stable-components, rolle-scoccola}. 
There is a large literature on consistency 
of density-based clustering methods. 
See for example 
\cite{cuevas-febrero-fraiman, 
chaudhuri-dasgupta-10, 
rinaldo-wasserman, 
chazal-guibas-oudot-skraba}.

\section{Preliminaries} \label{preliminaries}

We now give definitions and conventions that are used throughout the paper. 
For real numbers $a$ and $b$, the statement $a < b$ implies $a \neq b$.

\begin{definition}
Let $X$ be a metric space and $s > 0$. 
The \define{Vietoris--Rips complex} $\VR(X)(s)$ is the simplicial complex
\[
	\VR(X)(s) = \{ \{x_0, \dots, x_n\} \mid d_X(x_i, x_j) < s 
	\text{ for all } 0 \leq i, j \leq n \}.
\]
\end{definition}

We use $d_X(x_i, x_j) < s$ in this definition because the argument for 
\cref{lipschitz-deformation-retracts} does not work for the version of Vietoris--Rips 
defined with $d_X(x_i, x_j) \leq s$.

\begin{definition} \label{metric-probability-space}
A \define{metric probability space} consists of a metric space $X$ 
together with a Borel probability measure $\mu$ on $X$.
\end{definition}

For a metric space $X$ and $x \in X$, 
we write $B(x, s)$ for the open ball about $x$ of radius $s$.

\begin{definition}
Let $(X, \mu)$ be a metric probability space. 
The \define{uniform filtration} of $(X, \mu)$ is the two-parameter filtration of $X$ 
where, for $s > 0$ and $k \in [0, 1]$, 
$X_{(s, k)} \subseteq X$ is the sub-metric space 
$X_{(s, k)} = \{ x \in X : \mu(B(x, s)) \geq k \}$.
\end{definition}

The uniform filtration is the special case of the kernel filtration 
\cite[Def. 2.24]{rolle-scoccola}, 
where the kernel is chosen to be the uniform kernel \cite[Ex. 2.21]{rolle-scoccola}. 
Given a metric probability space, one can take the kernel filtration 
and then apply any functorial construction on metric spaces. 
For clustering, a natural choice is single-linkage \cite[Def. 2.25]{rolle-scoccola}. 
Applying Vietoris--Rips to the uniform filtration, 
we get an extension of the usual definition of degree-Rips. 
This is also considered in Scoccola's thesis \cite[Sec. 6.5]{scoccola-thesis}: 
the stability result mentioned in the introduction 
is a corollary of a stability result for the kernel filtration.

\begin{definition} \label{def-DR}
Let $(X, \mu)$ be a metric probability space.  
The \define{degree-Rips complex} $\DR(X, \mu)(s, k)$ is the simplicial complex 
$\VR(X_{(s, k)})(s)$.
\end{definition}

We now explain our conventions regarding the circle and annulus. 
For $\innerrad \geq 0$ we write 
$S^1_{\innerrad} = \{ p \in \bbR^2 : ||p || = \innerrad \}$, 
though we sometimes exclude the degenerate case $\innerrad = 0$. 
For $0 \leq \innerrad \leq \outerrad$ 
we write $A_{\innerrad, \outerrad} = \{ p \in \bbR^2 : \innerrad \leq ||p || \leq \outerrad \}$. 
Unless otherwise stated, we view these as metric spaces with the Euclidean metric.

Let $0 < \innerrad < \outerrad$, and let $\weight > 0$. 
We will consider a metric probability space $\wannulus$ 
that consists of the union of the annulus 
$A_{\innerrad, \outerrad}$ and the inner disc 
$\{ p \in \bbR^2 : || p || < \innerrad \}$, 
with a uniform measure on each piece, 
such that the measure of the inner disc is equal to $w$. 
In more detail, 
let $\wannulus$ be the metric probability space with underlying metric space 
$\{ p \in \bbR^2 : || p || \leq \outerrad \}$, 
and with probability measure $\mu$ given by integrating a density $f$, 
where $f(p) = \sdense = w / \pi \innerrad^2$ if $||p|| < \innerrad$ 
and  $f(p) = \ldense = (1 - w) / (\pi \outerrad^2 - \pi \innerrad^2)$ otherwise. 
We say that $\wannulus$ is a \define{weighted annulus} 
if $\sdense < \ldense$. 
Similarly, let $\wannuluszero$ 
be the metric probability space with underlying metric space 
$A_{\innerrad, \outerrad}$, 
and with probability measure $\mu$ given by integrating $f$, 
where $f(p) = \ldense = 1 / (\pi \outerrad^2 - \pi \innerrad^2)$.

\section{The Vietoris--Rips complexes of an annulus}

In this section we prove that the Vietoris--Rips complexes of an annulus 
$A_{\innerrad, \outerrad}$ are homotopy equivalent 
to the Vietoris--Rips complexes of the inner circle $S^1_{\innerrad}$. 
This was observed by Adamaszek--Adams 
\cite[Prop. 10.1]{adamaszek-adams-published}, 
who show that it follows from a result of Hausmann 
\cite[Prop. 2.2]{hausmann}.
The author overlooked this when writing the first draft of this paper, 
and it was pointed out by the reviewers. 

\cref{lipschitz-deformation-retracts} below is very similar to Hausmann's result. 
If $B$ is a deformation retract of a metric space $A$, 
then both results say that the Vietoris--Rips complexes of $A$ and $B$ are 
homotopy equivalent, 
provided the deformation retraction is sufficiently compatible with the metric. 
\cref{lipschitz-deformation-retracts} assumes slightly less about the 
deformation retraction than Hausmann's result, though this is a mild generalization, 
and may be known to experts. 
The exposition of the proof here is perhaps more detailed than Hausmann's, 
so it remains in the paper, in case it is of interest to some readers.

\begin{definition}
Let $A$ be a metric space and $B \subseteq A$. 
We say that $B$ is a \define{Lipschitz deformation retract} of $A$ if 
there is a continuous map $r \colon A \to B$ such that 
$r \circ i = \id_B$ where $i \colon B \into A$ is the inclusion, 
and there is a homotopy 
$H \colon A \times I \to A \, (\rel B)$ 
from $\id_A$ to $r$ such that 
for all $t \in I$, the map $H(- , t) : A \to A$ is $1$-Lipschitz.
\end{definition}

\begin{proposition} \label{lipschitz-deformation-retracts}
If $A$ is a metric space and $B$ is a Lipschitz deformation retract of $A$, 
then the inclusion $B \into A$ induces a homotopy equivalence 
$\VR(B)(s) \equiv \VR(A)(s)$ for all $s > 0$.
\end{proposition}

\begin{remark} \label{Psi}
Say $B$ is a Lipschitz deformation retract of $A$, 
and let $H \colon A \times I \to A$ be a homotopy 
as in the definition. 
Then for any $t \in I$ and any $s > 0$, 
there is a simplicial map 
$\Psimap_t \colon |\VR(A)(s)| \to |\VR(A)(s)|$
defined by the vertex map $x \mapsto H(x, t)$. 
These maps appear in the proof of \cref{lipschitz-deformation-retracts}, 
but note that the function
$|\VR(A)(s)| \times I \to |\VR(A)(s)|$ 
defined by $(p, t) \mapsto \Psimap_t(p)$ 
is not continuous in general.
\end{remark}

\begin{lemma} \label{straight-line-homotopy}
Let $K$ be a simplicial complex, let $X$ be compact, 
let $f, g \colon X \to |K|$ be continuous maps, 
and let $Z = \{ x \in X : f(x) = g(x) \}$. 
If for all $x \in X$, there is a simplex $\sigma \in K$ 
with $f(x), g(x) \in |\sigma|$, 
then $f$ and $g$ are homotopic $(\rel Z)$. 
\end{lemma}

\begin{proof}
We begin by proving the statement assuming that $K$ is finite. 
Define the homotopy $H \colon X \times I \to |K|$ as follows. 
For $x \in X$, write $f(x)$ and $g(x)$ in barycentric coordinates, 
$f(x) = \sum_{i} \alpha_i v_i$ and 
$g(x) = \sum_{j} \beta_j w_j$. 
For $t \in I$, let 
$H(x, t) = (1 - t) \cdot \sum_{i} \alpha_i v_i + t \cdot \sum_{j} \beta_j w_j$. 
We have $H(x, t) \in |K|$ since there is $\sigma \in K$ with $f(x), g(x) \in |\sigma|$. 
We now show that $H$ is continuous. 
For $\sigma \in K$, let $V_{\sigma} = f^{-1}(|\sigma|) \cap g^{-1}(|\sigma|)$. 
As $f$ and $g$ are continuous, $V_{\sigma}$ is closed in $X$, 
and $\mathcal{V} = \{ V_{\sigma} \times I : \sigma \in K \}$ 
is a finite, closed cover of $X \times I$. 
Now $H|_{V_{\sigma} \times I} \colon V_{\sigma} \times I \to |\sigma|$ 
is continuous for all $\sigma$, and therefore $H$ is continuous. 

Now we prove the general statement. 
As $X$ is compact, there are finite subcomplexes $K_f, K_g \subseteq K$ 
such that $f(X) \subseteq |K_f|$ and $g(X) \subseteq |K_g|$. 
Define 
$L = K_f \cup K_g \cup 
\{ \sigma \cup \tau : \sigma \in K_f, \tau \in K_g, \text{ and } \sigma \cup \tau \in K \}$. 
Then $L$ is a finite subcomplex of $K$ such that $f(X), g(X) \subseteq |L|$. 
For $x \in X$, let $\sigma \in K_f$ be the minimal simplex with $f(x) \in |\sigma|$ 
and let $\tau \in K_g$ be the minimal simplex with $g(x) \in |\tau|$. 
As $f(x), g(x)$ lie in a common simplex of $K$, we must have $\sigma \cup \tau \in K$, 
and thus $f(x), g(x)$ lie in a common simplex of $L$. 
Now the statement follows from the finite case.
\end{proof}

\begin{proof}[Proof of \cref{lipschitz-deformation-retracts}]
By assumption, there is $r \colon A \to B$ such that 
$r \circ i = \id_B$ where $i \colon B \into A$ is the inclusion, 
and there is a homotopy 
$H \colon A \times I \to A \, (\rel B)$ 
from $\id_A$ to $r$ such that 
for all $t \in I$, the map $H(- , t) : A \to A$ is $1$-Lipschitz. 
We show that $r_* \colon |\VR(A)(s)| \to |\VR(B)(s)|$ 
induces isomorphisms in $\pi_n$ for all $n \geq 0$, 
and the statement follows by Whitehead's theorem.
Since $r \circ i = \id_{B}$ it follows from functoriality that 
the induced maps on $\pi_n$ are surjective for all $n \geq 0$, 
and it remains to show they are injective.

We begin with $\pi_0$. 
For $x \in A$, observe that $[x] = [r(x)]$ in $\pi_0 \VR(A)(s)$, 
as $H(x, -) : I \to A$ is a path from $x$ to $r(x)$. 
Now, let $x, y \in A$, and say $r_*([x]) = r_*([y])$ in $\pi_0 \VR(B)(s)$. 
Then in $\pi_0 \VR(A)(s)$,
$[x] = [(i \circ r)(x)] = i_*(r_*([x])) = i_*(r_*([y])) = [(i \circ r)(y)] = [y].$

Now, let $b \in B$ be a choice of basepoint. 
Since we know $r$ induces an isomorphism on $\pi_0$, it suffices to consider basepoints in $B$. 
Say $f \colon I^n \to |\VR(A)(s)|$ is a continuous map representing an element of 
$\pi_n(|\VR(A)(s)|, b)$. As $I^n$ is compact, there is a finite subcomplex 
$K \subseteq \VR(A)(s)$ such that $f(I^n) \subseteq |K|$. 
Let $D = \max \{ \mathrm{diameter}(\sigma) : \sigma \in K \}$. 
As $K$ is finite, we have $D < s$. 
Write $\epsilon = s - D$. 
For $x \in A$, let $P_{x} = H(x, -) : I \to A$. 
As $I$ is compact, $P_{x}$ is uniformly continuous, 
and thus there is $\delta_{x} > 0$ such that 
$d_A(P_{x}(t), P_{x}(t')) < \epsilon$ when $|t - t'| < \delta_{x}$. 
Let $\delta = \min \{ \delta_{x} : x \text{ is a vertex of } K \}$, 
and choose $N$ such that $1 / N < \delta$. 
For $0 \leq m \leq N$, we write 
$\Psimap_{m} = \Psimap_{\frac{m}{N}}$ 
for the map induced by $H$, defined in \cref{Psi}.

We now show that, for any $0 \leq m < N$, 
$\Psimap_m \circ f \equiv \Psimap_{m+1} \circ f \, (\rel \partial I^n)$. 
By \cref{straight-line-homotopy}, it suffices to show that, for all $p \in I^n$, 
$(\Psimap_m \circ f)(p)$ and $(\Psimap_{m+1} \circ f)(p)$ lie in a common simplex of $\VR(A)(s)$. 
For this, choose $\sigma = \{ x_0, \dots, x_{\ell} \} \in K$ 
with $f(p) \in |\sigma|$; then 
$\{\Psimap_m(x_0), \dots, \Psimap_m(x_{\ell}), 
\Psimap_{m+1}(x_0), \dots, \Psimap_{m+1}(x_{\ell}) \}$ 
is the desired simplex. 
To see that it is indeed a simplex of $\VR(A)(s)$, observe that
\begin{align*}
	d_A (\Psimap_m(x_i), \Psimap_{m+1}(x_j)) &= 
	d_A (H(x_i, \tfrac{m}{N}), H(x_j, \tfrac{m+1}{N})) \\
	&\leq d_A (H(x_i, \tfrac{m}{N}), H(x_i, \tfrac{m+1}{N})) 
	+ d_A (H(x_i, \tfrac{m+1}{N}), H(x_j, \tfrac{m+1}{N})) \\
	&< \epsilon + d_A (x_i, x_j) \; \leq s \, .
\end{align*}

It follows that $f \equiv \Psimap_N \circ f \, (\rel \partial I^n)$. 
Note that $\Psimap_N = i_* \circ r_*$, 
where $r_* \colon |\VR(A)(s)| \to |\VR(B)(s)|$ is as above, 
and $i_* \colon |\VR(B)(s)| \to |\VR(A)(s)|$ 
is induced by the inclusion $i \colon B \into A$. 
Now say that $r_*([f]) = 0$ in $\pi_n \left( |\VR(B)(s)|, b \right)$. 
Then in $\pi_n \left( |\VR(A)(s)|, b \right)$, 
$[f] = [i_* \circ r_* \circ f] = i_*(r_*([f])) = 0$.
\end{proof}

\begin{corollary} \label{VR-of-annulus}
For any $0 \leq \innerrad \leq \outerrad$, 
the inclusion $S^1_{\innerrad} \into A_{\innerrad, \outerrad}$ 
induces a homotopy equivalence 
$\VR(S^1_{\innerrad})(s) \equiv \VR(A_{\innerrad, \outerrad})(s)$ 
for all $s > 0$.
\end{corollary}

\begin{proof}
The homotopy 
$H \colon A_{\innerrad, \outerrad} \times I \to A_{\innerrad, \outerrad}$ 
defined by 
$\left( (r, \theta), t \right) \mapsto 
\left( (1 - t) \cdot r + t \cdot \innerrad, \theta \right)$ 
shows that $S^1_{\innerrad}$ is a Lipschitz deformation retract of 
$A_{\innerrad, \outerrad}$.
\end{proof}

\section{Boundary curves in the degree-Rips parameter space}
\label{boundary-curves}

In this section we prove \cref{intro-theorem}. 
To motivate the approach, we first explain the basic idea 
for how to compute the homotopy type of $\DR(\wannulus)(s,k)$ 
for a particular choice of $s$ and $k$. 
We will show in this section that the subspace $\wannulus_{(s,k)}$ is an annulus; 
write $P$ for its inner radius.  
By \cref{VR-of-annulus}, 
$\DR(\wannulus)(s,k) \equiv \VR(S^1_{P})(s)$, 
and then one concludes by the Adamaszek--Adams calculation of the Vietoris--Rips 
complexes of the circle.

To begin, we need to compute the measure of an $s$-ball in $\wannulus$, 
and for this we need to know the area of the intersection of the $s$-ball with the annulus, 
and with the inner disc. 
We now briefly explain how to do this, though we omit most formulas for brevity. 

Let $R > 0$ and $s > 0$. 
We write $O$ for the origin in $\bbR^2$ 
and $C(p, r)$ for the circle centred at $p$ of radius $r$. 
Define a function $\alpha \colon [0, \infty) \to \bbR$ 
by letting $\alpha(c)$ be the area of the intersection 
$B\left(O, R \right) \cap B\left( (c, 0), s \right)$. 
We calculate $\alpha(c)$ using the usual formulas for the area of circular segments, 
but there are several cases. 
In \cref{c-s-space} on the left we show the $(c, s)$-space, 
with the curves that delimit the cases:
\begin{align}
	c + s &= R \\
	s - c &= R \\
	c - s &= R \\
	c^2 + s^2 &= R^2 \\
	s^2 - c^2 &= R^2
\end{align}
If $(c, s)$ is outside the region bounded by the lines $1, 2, 3$, 
then the circles $C(O, R)$ and $C((c, 0), s)$ do not intersect, and $\alpha$ is constant. 
If $(c, s)$ is inside this region, then the circles intersect in two points: 
if $(c, s)$ is to the left of the circle $4$, 
then the centre $(c, 0)$ is to the left of the chord connecting these points of intersection, 
otherwise it is to the right; 
if $(c, s)$ is to the left of the hyperbola $5$, 
then the origin is to the right of this chord, 
otherwise it is to the left. 
In each region, we calculate $\alpha(c)$ as a sum of areas of circular segments, 
or their complements.

\begin{figure}[hbtp]
	\centering
	\begin{minipage}[b]{0.45\textwidth}
		\raisebox{70pt}{
		\begin{overpic}[width=0.9\textwidth]{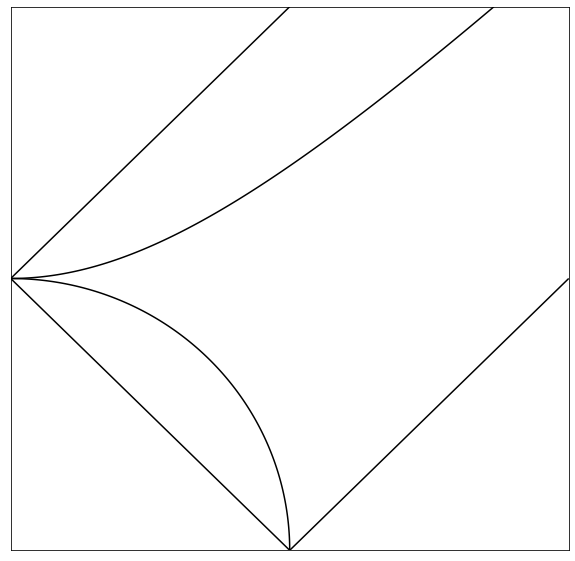}
		\put (30,10) {1}
		\put (33,90) {2}
		\put (92,32) {3}
		\put (38,38) {4}
		\put (60,70) {5}
		\put (49, -5) {R}
		\put (-5, 47){R}
		\end{overpic}}
	\end{minipage}
	\hfill
	\begin{minipage}[b]{0.45\textwidth}
		\begin{overpic}[width=0.9\textwidth]{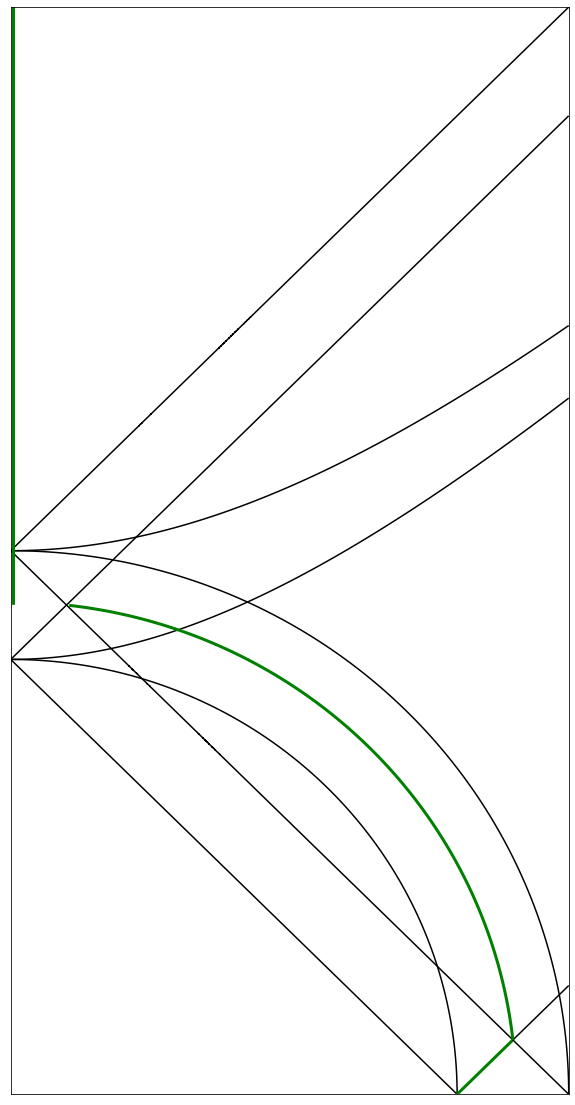}
		\put (25.5,10.5) {1R}
		\put (30.5,15.5) {1Q}
		\put (35,78.5) {2R}
		\put (35,89) {2Q}
		\put (29.5,29.5) {4R}
		\put (36.5,36.5) {4Q}
		\put (40,53) {5R}
		\put (40,60.5) {5Q}
		\put (48,5) {3R}
		\put (40.5, -2) {R}
		\put (50.5, -2) {Q}
		\put (-3, 39.5) {R}
		\put (-3, 49.5) {Q}
		\end{overpic}
	\end{minipage}
	\caption{Let $R > 0$. 
	We can think of $\alpha_R$ as a function of two variables $c$ and $s$. 
	On the left, we show the domain of $\alpha_R$, with 
	$c$ on the horizontal axis and $s$ on the vertical axis. 
	In each region we compute $\alpha_R$ using a sum of areas of circular segments, 
	or their complements. 
	To compute $\gamma$, we need to compute $\alpha_R$ and $\alpha_Q$. 
	On the right we plot the domain of $\gamma$ as a function of $c$ and $s$, 
	with	 $R = 0.4$ and $Q = 0.5$. 
	In green, we plot $\omega(s)$ with $w=0$. 
	In this case 
	$\omega(s)$ is the value of $c$ that maximizes the area of the intersection 
	between the annulus $A_{\innerrad, \outerrad}$ 
	and the $s$-ball centred at $(c,0)$.}
	\label{c-s-space}
\end{figure}
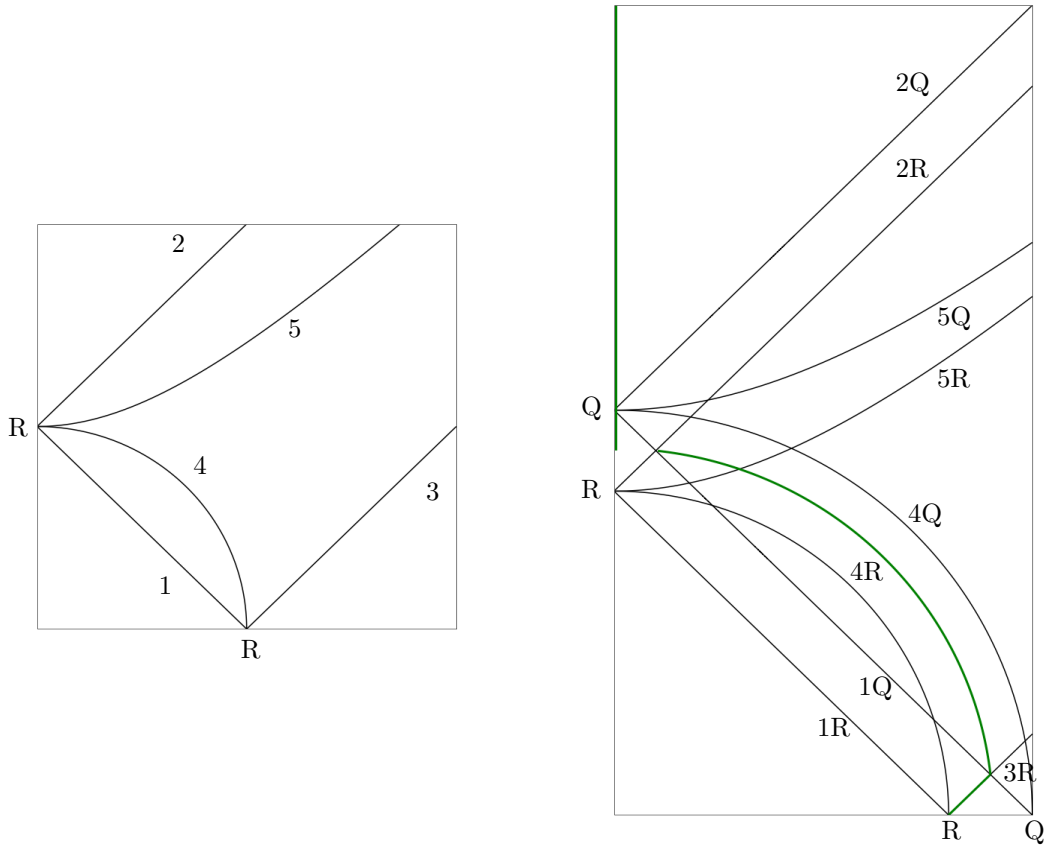

In order to prove the main theorem, 
we need to understand how the area of a ball in $\wannulus$ changes as its centre varies. 
For this it will be helpful to understand the derivative of $\alpha$. 
This seems complicated to compute directly from the formula 
for the area of a circular segment, 
but instead we can compute the derivative using a geometric argument. 
Define $y \colon [0, \infty) \to \bbR$ by
\[
	y(c) = 
	\begin{cases}
	0 & \text{if $C(O, R) \cap C((c, 0), s) = \emptyset$} \\
	\max \{ x_2 : (x_1, x_2) \in C(O, R) \cap C((c, 0), s) \} & \text{else}
	\end{cases}
\]
If the circles intersect in exactly two points, 
then $y(c)$ is simply the $y$-coordinate of these points, with positive sign, 
and then $2y(c)$ is the length of the chord connecting these points of intersection.

\begin{lemma} \label{alpha-prime}
The map $\alpha$ is continuously differentiable and $\alpha'(c) = -2y(c)$.
\end{lemma}

\begin{proof}
Since $\alpha$ can be computed by summing the areas of circular segments, 
the formula for this area shows that $\alpha$ is differentiable. 
By definition, 
\[
	\alpha'(c) = \lim_{\delta \to 0} \frac{\alpha(c + \delta) - \alpha(c)}{\delta} \; .
\]
Now, $\alpha(c + \delta) - \alpha(c)$ is the (signed) area of the subset of 
$B\left( O, R \right)$ between the circles 
$C((c, 0), s)$ and $C((c + \delta, 0), s)$. 
And, $-2y(c) \cdot \delta$ approximates this area, in the sense that 
we have the following inequality for all $s, c$ and small enough $\delta$:
\[
	| \left( \alpha(c+\delta) - \alpha(c) \right) + 2y(c) \cdot \delta | 
	\leq 2 \cdot |\delta| \cdot |y(c) - y(c + \delta)| \; .
\]
So, 
\begin{align*}
	| \alpha'(c) + 2y(c) | 
	&\leq \lim_{\delta \to 0} \frac{2 \cdot |\delta| \cdot |y(c) - y(c + \delta)|}{|\delta|} \\
	&= \lim_{\delta \to 0} 2 \cdot |y(c) - y(c+\delta)| = 0 \; .
\end{align*}
\end{proof}

Let $0 < \innerrad < \outerrad$. 
We now consider the area of the intersection of the ball $B\left( (c, 0), s \right)$ 
with the annulus $\annulus = \{ p \in \bbR^2 : \innerrad \leq || p || \leq \outerrad\}$. 
Define a function $\gamma \colon [0, \outerrad] \to \bbR$ 
by letting $\gamma(c)$ be the area of the intersection 
$\annulus \cap B\left( (c, 0), s \right)$. 
We change notation in order to relate $\gamma$ with $\alpha$: 
instead of $\alpha$ we write $\alpha_R$, 
where $R$ is the radius of the circle centred at the origin. 
With this notation, $\gamma = \alpha_{\outerrad} - \alpha_{\innerrad}$.
We calculate $\gamma$ using this formula, though now there are more cases; 
see \cref{c-s-space} on the right for the curves delimiting the cases. 

Again we want to understand $\gamma'$. 
The interesting case is when the ball $B\left( (c, 0), s \right)$ 
intersects both circles $C(O, \innerrad)$ and $C(O, \outerrad)$. 
This happens when $(c, s)$ is in the region bounded by the lines:
\begin{align*}
	c + s &= Q \tag{1Q} \\
	s - c &= R \tag{2R} \\
	c - s &= R \tag{3R}
\end{align*}
We now collect together the facts we will need about $\gamma$:

\begin{lemma} \label{gamma-prime}
Say $c + s > Q$ and $s - c < R$ and $c - s < R$. Then,
\vspace{2pt}
\begin{itemize}
	\item $\gamma'(c) = 0 \text{ \; if and only if \; } c^2 + s^2 = 
	\tfrac{1}{2}(\innerrad^2 + \outerrad^2)$
	\vspace{2pt}
	\item $\gamma'(c) > 0 \text{ \; if and only if \; } c^2 + s^2 < 
	\tfrac{1}{2}(\innerrad^2 + \outerrad^2)$
	\vspace{2pt}
	\item $\gamma'(c) < 0 \text{ \; if and only if \; } c^2 + s^2 > 
	\tfrac{1}{2}(\innerrad^2 + \outerrad^2)$.
	\vspace{4pt}
\end{itemize}
Moreover, if we have $c \leq c'$ with 
$c' + s > Q$ and $s - c' < R$ and $c' - s < R$ 
and $\gamma'(c), \gamma'(c') > 0$, 
then $\gamma'(c) \geq \gamma'(c')$.
\end{lemma}

\begin{proof}
As $\gamma'(c) = \alpha_{\outerrad}'(c) - \alpha_{\innerrad}'(c)$, 
we have $\gamma'(c) = -2y_{\outerrad}(c) + 2y_{\innerrad}(c)$ 
by \cref{alpha-prime}. 
So, 
$\gamma'(c) = 0$ if and only if $y_{\outerrad}(c) = y_{\innerrad}(c)$; and 
$\gamma'(c) > 0$ if and only if $y_{\outerrad}(c) < y_{\innerrad}(c)$; and 
$\gamma'(c) < 0$ if and only if $y_{\outerrad}(c) > y_{\innerrad}(c)$.

Now, $C(O, R)$ and $C((c, 0), s)$ intersect in two points, 
and the $x$-coordinate of both points is given by 
$x_{\innerrad}(c) = (c^2 + \innerrad^2 - s^2) / 2c$, 
and then $y_{\innerrad}(c) = \sqrt{\innerrad^2 - x_{\innerrad}(c)^2}$. 
Similarly we compute $y_{\outerrad}(c)$. 
Now, a little algebra shows that $y_{\outerrad}(c) = y_{\innerrad}(c)$ 
if and only if $c^2 + s^2 = \tfrac{1}{2}(\innerrad^2 + \outerrad^2)$; 
and $y_{\outerrad}(c) < y_{\innerrad}(c)$ if and only if 
$c^2 + s^2 < \tfrac{1}{2}(\innerrad^2 + \outerrad^2)$; 
and $y_{\outerrad}(c) > y_{\innerrad}(c)$ if and only if 
$c^2 + s^2 > \tfrac{1}{2}(\innerrad^2 + \outerrad^2)$.

To prove the last statement of the lemma, 
we must show that 
$y_{\innerrad}(c) - y_{\outerrad}(c) \geq y_{\innerrad}(c') - y_{\outerrad}(c')$. 
Towards a contradiction, assume 
$y_{\innerrad}(c) - y_{\outerrad}(c) < y_{\innerrad}(c') - y_{\outerrad}(c')$. 
Then a little more algebra shows that 
$c'^2(\innerrad^2 + \outerrad^2 - 2s^2) < c^2(\innerrad^2 + \outerrad^2 - 2s^2)$. 
As we assume $\gamma'(c) > 0$, we have
$c^2 + s^2 < \tfrac{1}{2}(\innerrad^2 + \outerrad^2)$, 
and therefore 
$s^2 < \tfrac{1}{2}(\innerrad^2 + \outerrad^2)$. 
So, we conclude $c' < c$, a contradiction.
\end{proof}

We are ready to show that the subspace $\wannulus_{(s,k)}$ is an annulus, 
for any choice of $s$ and $k$. This will follow from the next lemma.

Now, for $s > 0$, 
let $\nu \colon [0, \outerrad] \to [0, 1]$ be the function 
$\nu(c) = \mu \left(B\left((c,0), s \right)\right)$, 
where $\mu$ is the measure on $\wannulus$. 
We write $\nu_s$ if it is necessary to specify $s$. 
Since $\nu$ is a linear combination of $\gamma$ and $\alpha_{\innerrad}$, 
we have already seen how to calculate $\nu$. 
Define $\omega \colon (0, \infty) \to [0, \outerrad]$ as follows. 
For any $s > 0$, 
let $M_s = \max_{c \in [0, \outerrad]} \nu_s(c)$. 
As $\nu_s$ is continuous, $\nu_s^{-1}(M_s) \subseteq [0, \outerrad]$ 
is non-empty and closed, and we let $\omega(s) = \min(\nu_s^{-1}(M_s))$. 

\begin{lemma} \label{shape-of-nu}
For any $s > 0$, $\nu_s$ is non-decreasing on $[0, \omega(s)]$ 
and non-increasing on $[\omega(s), \outerrad]$.
\end{lemma}

\begin{proof}
There are three cases: 
$(1)$ $0 < s \leq \frac{1}{2}(\outerrad - \innerrad)$; 
$(2)$ $\frac{1}{2}(\outerrad - \innerrad) < s < \frac{1}{2}(\outerrad + \innerrad)$; and 
$(3)$ $\frac{1}{2}(\outerrad + \innerrad) \leq s$. 
See \cref{c-s-space} for an idea of how $\omega$ behaves in the three cases.

Case $(1)$ is straightforward. 
On $[0, \innerrad - s]$ $\nu$ is constant, 
it is strictly increasing on $[\innerrad - s, \innerrad + s]$, 
and we have $\nu(c) = M_s$ for any 
$c \in [\innerrad + s, \outerrad - s]$; so $\omega(s) = \innerrad + s$. 
On $[\outerrad - s, \outerrad]$ $\nu$ is strictly decreasing. 

We now consider Case $(2)$. 
On $[0, \innerrad - s]$ $\nu$ is constant. 
In the region bounded by lines $1\innerrad, 1\outerrad, 2\innerrad$, 
$\nu$ is strictly increasing. 
If $c$ is to the left of line $2\innerrad$, 
then $\nu$ is constant. 
The interesting case is that $c$ is to the right of line $1\outerrad$: 
$c \geq \outerrad - s$. We will show that $\omega(s)$ is in this region, 
and that $\nu' > 0$ on $[\outerrad - s, \omega(s))$, 
$\nu'(\omega(s)) = 0$, and $\nu' < 0$ on $(\omega(s), \outerrad]$. 

Let $z = \sqrt{\frac{1}{2}(\innerrad^2 + \outerrad^2) - s^2}$. 
Then by \cref{gamma-prime}, 
$\gamma'(c) < 0$ for all $c \in (z, \outerrad]$, 
and $\alpha_{\innerrad}'(c) \leq 0$, 
so $\nu'(c) < 0$. 
When $c = \outerrad - s$, so that $c$ is on line $1Q$, 
we have already seen that $\nu'(c) \geq 0$. 
So, as $\nu'$ is continuous, it must have a zero on $[\outerrad - s, z]$. 
We now show that $\nu'$ has at most one zero on this interval, 
and it follows that $\omega(s)$ is this zero.

Let $c \leq c'$ in $[\outerrad - s, z]$ such that 
$\nu'(c) = \nu'(c') = 0$. 
As $\nu = \sdense \cdot \alpha_{\innerrad} + \ldense \cdot \gamma$, 
$\nu'(c) = 0$ implies that $\gamma'(c) = \frac{-\sdense}{\ldense} \cdot \alpha_{\innerrad}'(c)$. 
We have $\gamma = \alpha_{\outerrad} - \alpha_{\innerrad}$, 
so $\nu'(c) = 0$ implies that 
$(1 - \frac{\sdense}{\ldense}) \cdot \alpha_{\innerrad}'(c) = \alpha_{\outerrad}'(c)$. 
By \cref{gamma-prime}, since $c \leq c'$ in $[\outerrad - s, z]$, 
we have $\gamma'(c) \geq \gamma'(c')$. So,
\begin{align*}
	\tfrac{-\sdense}{\ldense} \cdot \alpha_{\innerrad}'(c) &\geq 
	\tfrac{-\sdense}{\ldense} \cdot \alpha_{\innerrad}'(c') \\
	(1 - \tfrac{\sdense}{\ldense}) \cdot \alpha_{\innerrad}'(c) &\leq 
	(1 - \tfrac{\sdense}{\ldense}) \cdot \alpha_{\innerrad}'(c') \\
	\alpha_{\outerrad}'(c) &\leq \alpha_{\outerrad}'(c') \\
	y_{\outerrad}(c) &\geq y_{\outerrad}(c') .
\end{align*}
As $y_{\outerrad}$ is strictly increasing 
on $[\outerrad - s, \sqrt{\outerrad^2 - s^2}]$ 
we have $c = c'$. 
This finishes Case $(2)$. 

Case $(3)$ is straightforward; $\omega(s) = 0$. 
If $c$ is to the left of line $2Q$, or 
$c$ is in the region bounded by lines $1Q$ and $2R$, then $\nu$ is constant. 
If $c$ is in the region bounded by lines $1Q, 2R$ and $2Q$, 
then $\nu$ is strictly decreasing. 
If $c$ is to the right of line $2R$ then $\nu'(c) < 0$ 
by \cref{gamma-prime}.
\end{proof}

Note that this proof also shows how to compute $\omega(s)$. 
If $0 < s \leq \frac{1}{2}(\outerrad - \innerrad)$, 
then $\omega(s) = \innerrad + s$. 
If $\frac{1}{2}(\outerrad - \innerrad) < s < \frac{1}{2}(\outerrad + \innerrad)$, 
then $\omega(s)$ is defined implicitly by the equation 
$(\ldense - \sdense) \cdot y_{\innerrad}(c) = \ldense \cdot y_{\outerrad}(c)$. 
In this case, if $w=0$, then this last equation simplifies to 
$c^2 + s^2 = \tfrac{1}{2}(\innerrad^2 + \outerrad^2)$. 
If $\frac{1}{2}(\outerrad + \innerrad) \leq s$, $\omega(s) = 0$. 

We are almost ready to define the maps $\varphi_{\ell}$ and prove \cref{intro-theorem}. 
In order to make use of the Adamaszek--Adams calculation of the Vietoris--Rips 
complexes of the circle, we need to relate the Vietoris--Rips complexes 
of a circle with the Euclidean distance to the Vietoris--Rips complexes 
of the circle with geodesic distance. 
For this we use the following lemma, 
whose proof is straightforward.

\begin{lemma} \label{changing-the-metric}
Let $(X_1, d_1)$ and $(X_2, d_2)$ be metric spaces, 
let $D_1 = \Im(d_1) \subseteq \bbR_{\geq 0}$ and $D_2 = \Im(d_2) \subseteq \bbR_{\geq 0}$, 
and say there is a bijection $f \colon X_1 \to X_2$ 
and an order-preserving bijection $f_d \colon D_1 \to D_2$ 
such that $f_d \circ d_1 = d_2 \circ (f \times f)$. 
Then $f$ induces an isomorphism 
$\VR(X_1, d_1)(s) \iso \VR(X_2, d_2)(f_d(s))$ for any $s \in D_1$.
\end{lemma}

For $r > 0$, we write $(S^1_r, \dg)$ for the circle of radius $r$ 
equipped with the geodesic metric, 
and $(S^1_r, \dE)$ for the circle of radius $r$ 
equipped with the Euclidean metric. 
Define $\sigma_r \colon [0, 2r] \to [0, \pi r]$ by 
$\sigma_r(t) = 2r \arcsin \left(\frac{t}{2r}\right)$. 
If $p,q \in S^1_r$, then 
$\sigma_r(\dE(p,q)) = \dg(p,q)$.

By Adamaszek--Adams \cite[Theorem 7.4]{adamaszek-adams-published}, 
\[
	\VR(S^1_{\frac{1}{2\pi}}, \dg)(s) \equiv S^{2\ell+1} \text{ \; for \; } 
	\tfrac{\ell}{2\ell + 1} < s \leq \tfrac{\ell + 1}{2\ell + 3}, \; \ell = 0, 1, \dots .
\]
And, if $\frac{\ell}{2\ell + 1} < s \leq s' \leq \frac{\ell + 1}{2\ell + 3}$, 
then the inclusion $\VR(S^1_{\frac{1}{2\pi}}, \dg)(s) \into \VR(S^1_{\frac{1}{2\pi}}, \dg)(s')$ 
is a homotopy equivalence.

In order to define the maps $\varphi_{\ell}$, 
we need to find, for any Vietoris--Rips parameter value $s > 0$, 
the radius $r$ such that $\VR(S^1_{r}, \dE)(s)$ 
is isomorphic to $\VR(S^1_{\frac{1}{2\pi}}, \dg)(\tfrac{\ell}{2\ell+1})$. 
So, for any integer $\ell > 0$, 
let $\rho_{\ell} \colon (0, \infty) \to (0, \infty)$ be defined by 
\[
	\rho_{\ell}(s) = \frac{s}{2\sin(\frac{\pi \ell}{2\ell + 1})} \; .
\]
Then, for any $s > 0$ we have
\[
	\tfrac{\ell}{2\ell + 1} = 
	\frac{\sigma_{\rho_{\ell}(s)} (s)}{2 \pi \rho_{\ell} (s)} \, ,
\]
and therefore by \cref{changing-the-metric}, 
we have
\[
	\VR(S^1_{\rho_{\ell}(s)}, \dE)(s) \iso 
	\VR(S^1_{\frac{1}{2\pi}}, \dg)(\tfrac{\ell}{2\ell+1}) \; .
\]
Similarly, define $\rho_{\infty} \colon (0, \infty) \to (0, \infty)$ 
by $\rho_{\infty}(s) = s / 2$.

We can now define the maps $\varphi_{\ell} \colon (0, \infty) \to [0,1]$ 
for $\ell = 0, 1, 2, \dots, \infty$. 
For the case $\ell = 0$, we let $\varphi_0(s) = \nu_s(\omega(s)) = M_s$. 
For $\ell > 0$, let
\[
	\varphi_{\ell}(s) = \nu_s\left(\min\left(\rho_{\ell}(s), \omega(s)\right)\right) \; .
\]
Note that, by \cref{shape-of-nu}, 
for any $s > 0$ and $0 \leq \ell < \ell' \leq \infty$, 
we have $\varphi_{\ell}(s) \geq \varphi_{\ell'}(s)$.

\begin{proof}[Proof of \cref{intro-theorem}]
Write $\wann = \wannulus$. 
If $k > \varphi_0 (s) = M_s$, 
then $\wann_{(s, k)} = \emptyset$, 
so that $\DR(\wann)(s, k) = \emptyset$. 
Next, we show that if $\wann_{(s, k)}$ is non-empty, 
then it is an annulus. 
Now, 
\begin{align*}
	\wann_{(s, k)} &= \{ p \in \wann \mid \mu(B(p, s)) \geq k \} \\
	&= (\nu_s \circ || - ||)^{-1} \left([k, 1]\right)
\end{align*}
which is closed as $\nu_s$ and $|| - ||$ are continuous. 
It suffices to show that 
$\nu_s^{-1} \left([k, 1]\right) \subset [0, \outerrad]$ is an interval, 
and this follows from \cref{shape-of-nu}. 

Now, say that 
$0 < \ell < \infty$ and $s > 0$ and $k \in [0,1]$ are such that  
$\varphi_{\ell}(s) > k > \varphi_{\ell+1}(s)$. 
Let $P$ be the left endpoint of the interval $\nu_s^{-1} \left([k, 1]\right)$, 
so that $\wann_{(s, k)}$ is an annulus with inner radius $P$. 
We show now that $\rho_{\ell+1}(s) < P < \rho_{\ell}(s)$.

As $\varphi_{\ell}(s) \neq \varphi_{\ell+1}(s)$ 
and $\rho_{\ell}(s) > \rho_{\ell+1}(s)$, 
we have 
$\rho_{\ell+1}(s) < \omega(s)$ and 
$\varphi_{\ell+1}(s) = \nu_s(\rho_{\ell+1}(s))$. 
As $k > \varphi_{\ell+1}(s) = \nu_s(\rho_{\ell+1}(s))$ 
we have $\rho_{\ell+1}(s) \notin \nu_s^{-1} \left([k, 1]\right)$; 
as $\omega(s) \in \nu_s^{-1} \left([k, 1]\right)$ 
and $\rho_{\ell+1}(s) < \omega(s)$, we have 
$\rho_{\ell+1}(s) < P$, as desired. 
By continuity of $\nu_s$, 
there is $r \in (\rho_{\ell+1}(s), P]$ 
with $\nu_s(r) = k$. By definition of $P$, 
we have $P \leq r$, and thus $P = r$ and $\nu_s(P) = k$. 
Since $\varphi_{\ell}(s) > k$, we have $P < \rho_{\ell}(s)$.

Now, by \cref{VR-of-annulus}, 
the inclusion $S^1_P \into \wann_{(s, k)}$ induces a homotopy equivalence
$\VR(S^1_P)(s) \equiv \VR(\wann_{(s, k)})(s)$. 
By \cref{changing-the-metric}, 
$\VR(S^1_P)(s) \iso \VR(S^1_{\frac{1}{2\pi}}, \dg)(\frac{\sigma_{P} (s)}{2 \pi P})$.

As $\rho_{\ell+1}(s) < P < \rho_{\ell}(s)$, 
we have
\[
	\tfrac{\ell+1}{2\ell + 3} = 
	\frac{\sigma_{\rho_{\ell+1}(s)} (s)}{2 \pi \rho_{\ell+1} (s)} > 
	\frac{\sigma_{P} (s)}{2 \pi P} >
	\frac{\sigma_{\rho_{\ell}(s)} (s)}{2 \pi \rho_{\ell} (s)} = 
	\tfrac{\ell}{2\ell + 1}
\]
So that 
$\DR(\wann)(s,k) = \VR(\wann_{(s, k)})(s) \equiv \VR(S^1_P)(s) \equiv S^{2\ell+1}$ 
by \cite[Theorem 7.4]{adamaszek-adams-published}.

If $s$ and $k$ are such that $\varphi_{\infty}(s) > k$, 
then we have seen that 
$\wann_{(s, k)}$ is an annulus, and again we write $P$ for the inner radius. 
Then one checks that $P < \rho_{\infty}(s) = s / 2$, 
and so $\VR(S^1_P)(s)$ is contractible.

The claim that inclusions 
$\DR(\wann)(s,k) \into \DR(\wann)(s',k')$ 
are homotopy equivalences whenever $(s,k)$ and $(s',k')$ 
both lie between $\varphi_{\ell}$ and $\varphi_{\ell+1}$ follows from
\cref{VR-of-annulus} 
and the statement in \cite[Theorem 7.4]{adamaszek-adams-published} 
about inclusions of Vietoris--Rips complexes.
\end{proof}

\section{The annulus without outliers}
\label{the-annulus-without-outliers}

We now consider the case $w = 0$, 
when the measure of the inner disc $\{ p \in \bbR^2 : || p || < \innerrad \}$ is zero. 
The measure of an $s$-ball $B(p, s)$ is not much changed from the case 
where $w$ is small but non-zero. 
However, the degree-Rips complexes of $\wannuluszero$ 
exhibit different behavior from the case $w > 0$, 
because now the vertices of the degree-Rips complexes 
must lie in the annulus $A_{\innerrad, \outerrad}$. 
In this section, we modify the constructions of \cref{boundary-curves} 
accordingly, and then prove the analogue of \cref{intro-theorem} in this case.

For $s > 0$, let $\tilde{\nu_s} \colon [\innerrad, \outerrad] \to [0,1]$ 
be defined by $\tilde{\nu_s}(c) = \mu \left(B\left((c,0), s\right)\right)$. 
Define $\tilde{\omega} \colon (0, \infty) \to [\innerrad, \outerrad]$ as follows. 
For any $s > 0$, 
let $M_s = \max_{c \in [\innerrad, \outerrad]} \tilde{\nu_s}(c)$. 
As $\tilde{\nu_s}$ is continuous, $\tilde{\nu_s}^{-1}(M_s) \subseteq [\innerrad, \outerrad]$ 
is non-empty and closed, and we let $\tilde{\omega}(s) = \min(\tilde{\nu_s}^{-1}(M_s))$.

As before, $\tilde{\varphi_0} \colon (0, \infty) \to [0,1]$ 
is defined as $\tilde{\varphi_0}(s) = M_s = \tilde{\nu_s}(\tilde{\omega}(s))$. 
But now, for $0 < \ell \leq \infty$, 
we define 
$\tilde{\varphi_{\ell}} \colon (0, \infty) \to [0,1]$ by
\[
	\tilde{\varphi_{\ell}}(s) =  
	\begin{cases}
	0 & \text{ \; if \; } \rho_{\ell}(s) \leq \innerrad \\
	\tilde{\nu_s}\left(\min\left(\rho_{\ell}(s), \tilde{\omega}(s)\right)\right) 
	& \text{ \; else}
	\end{cases}
\]
Note that the $\tilde{\varphi_{\ell}}$ need not be continuous.

\begin{theorem}
For any $s > 0$ and any $k \in [0,1]$,
\[
	\DR(\wannuluszero)(s,k) \equiv 
	\begin{cases}
	\emptyset & \text{ \; if \; } k > \tilde{\varphi_0}(s) \\
	S^{2\ell+1} & \text{ \; if \; } 
	\tilde{\varphi_{\ell}}(s) > k > \tilde{\varphi_{\ell+1}}(s) 
	\text{ \; for } \ell \neq \infty \\
	* & \text{ \; if \; } \tilde{\varphi_{\infty}}(s) > k
	\end{cases}
\]
Moreover, if $0 < \ell < \infty$ and $0 < s \leq s'$ and $0 \leq k' \leq k \leq 1$ 
are such that
\[
	\tilde{\varphi_{\ell}}(s) > k > \tilde{\varphi_{\ell+1}}(s) \text{ \; and \; } 
	\tilde{\varphi_{\ell}}(s') > k' > \tilde{\varphi_{\ell+1}}(s') \, ,
\]
then the inclusion 
$\DR(\wannuluszero)(s,k) \into \DR(\wannuluszero)(s',k')$ 
is a homotopy equivalence.
\end{theorem}

\begin{proof}
The proof is quite similar to the proof of \cref{intro-theorem}. 
If $s > 0$ and $k \in [0,1]$ are such that 
$\tilde{\varphi_{\ell}}(s) > k > \tilde{\varphi_{\ell+1}}(s)$, 
then, arguing as before, 
$\wannuluszero_{(s,k)}$ is an annulus with inner radius $P$ such that 
$\rho_{\ell+1}(s) \leq P < \rho_{\ell}(s)$. 
Therefore,
\[
	\tfrac{\ell+1}{2\ell + 3} = 
	\frac{\sigma_{\rho_{\ell+1}(s)} (s)}{2 \pi \rho_{\ell+1} (s)} \geq 
	\frac{\sigma_{P} (s)}{2 \pi P} >
	\frac{\sigma_{\rho_{\ell}(s)} (s)}{2 \pi \rho_{\ell} (s)} = 
	\tfrac{\ell}{2\ell + 1}
\]
So that
\[
	\DR(\wannuluszero)(s,k) = 
	\VR(\wannuluszero_{(s, k)})(s) 
	\equiv \VR(S^1_P)(s) \equiv S^{2\ell+1}
\]
again by \cite[Theorem 7.4]{adamaszek-adams-published}. 
The claim about inclusions of degree-Rips complexes is proved in the same way as before.
\end{proof}

\section{Conclusions}

In various experiments, and in this paper, we have observed the following behavior. 
If there is a strong topological signal in data, 
and this appears somewhere in the parameter space of degree-Rips, 
then if one adds outliers, the topological signal is still visible 
(i.e., prominent) 
in degree-Rips, but at a different location in the parameter space, 
where the values of the Rips parameter are smaller.

The main interest of the calculation in this paper is that, in this case, 
it is possible to say precisely how the location of the signal changes 
in the degree-Rips parameter space. 
We now briefly mention one reason why we would like to understand this in more general settings. 
If one is interested in taking one-parameter slices of degree-Rips 
(e.g., for computing a barcode, or for clustering as in 
robust single-linkage \cite{chaudhuri-dasgupta-10} 
or $\gamma$-linkage \cite{rolle-scoccola}), 
then choosing the slice is tricky in practice. 
But it seems that, both for computational reasons and to maximize robustness, 
one wants to choose a slice through ``small'' values of the Rips parameter. 
A satisfactory understanding of the robustness of degree-Rips may shed light on this.

There are several directions in which one could try to extend the results of this paper. 
Of course it would be interesting to consider measures supported on more complicated spaces, 
perhaps seeking only partial calculations or approximations. 
One could also consider other models for outliers. 
For example, one could take a convolution with a kernel 
(as in \cite[Section 2.1]{rinaldo-wasserman}), 
rather than mixing with a uniform measure. 
Finally, a reviewer posed the following question: 
is there a density on the disc that is rotationally invariant and monotone in the radius 
such that the uniform filtration at some parameter $(s,k)$ is not an annulus? 

\bibliography{degree-Rips_annulus}

\end{document}